\documentclass[a4paper,12pt, reqno]{amsart} \usepackage{latexsym}
\usepackage{amssymb,amsfonts,amsmath,mathrsfs}
\begin{document}
\title{A note on the fourth moment of Dirichlet $L$-functions}
\author{H. M. Bui \and D. R. Heath-Brown}
\address{Mathematical Institute, University of Oxford, OXFORD, OX1 3LB}
\email{hung.bui@maths.ox.ac.uk}
\address{Mathematical Institute, University of Oxford, OXFORD, OX1 3LB}
\email{rhb@maths.ox.ac.uk}

\begin{abstract}
We prove an asymptotic formula for the fourth power mean of Dirichlet
$L$-functions averaged over primitive characters to modulus $q$ and over
$t\in[0,T]$ which is particularly effective when $q\ge T$. In this
range the correct order of magnitude was not previously known.
\end{abstract}
\maketitle

\section{Introduction}

For $\chi$ a Dirichlet character (mod $q$), the moments of $L(s,\chi)$
have many applications, for example to the distribution of primes in
the arithmetic progressions to modulus $q$. The asymptotic formula of
the fourth power moment in the $q$-aspect has been obtained by
Heath-Brown [\textbf{\ref{H-B1}}], for $q$ prime, and more recently by
Soundararajan
[\textbf{\ref{S}}] for general $q$. Following Soundararajan's work, 
Young [\textbf{\ref{Y}}] pushed the
result much further by computing the fourth moment for prime moduli
$q$ with a power saving in the error term. The problem essentially
reduces to the analysis of a particular divisor sum. To this end,
Young used various techniques to estimate the off-diagonal terms. 

In the case that the $t$-aspect is also included, a result of
Montgomery [\textbf{\ref{M}}] states that  
\begin{equation*}
\sum_{\chi(\textrm{mod}\ q)}{\!\!\!\!\!\!}^{\textstyle{*}}\
\int_{0}^{T}|L({\scriptstyle{\frac{1}{2}}}+it,\chi)|^4dt\ll\varphi(q)T(\log
qT)^4
\end{equation*}
for $q,T\ge 2$, where $\sum_{\chi \ \!\!(\textrm{mod}\ \!q)}^{*}$ 
indicates that the
sum is restricted to the primitive characters modulo $q$. As we shall
see, the upper bound is too large by a factor $(q/\varphi(q))^5$.  A
second result of relevance is due to Rane [\textbf{\ref{Rane}}].  After
correcting a misprint it states that
\begin{eqnarray*}
&&\lefteqn{\sum_{\chi(\textrm{mod}\ q)}{\!\!\!\!\!\!}^{\textstyle{*}}\
\int_{T}^{2T}|L({\scriptstyle{\frac{1}{2}}}+it,\chi)|^{4}dt}\\
&&\qquad\qquad=\frac{\varphi^{*}(q)T}{2\pi^{2}}\prod_{p|q}
\frac{(1-p^{-1})^3}{(1+p^{-1})}(\log 
qT)^{4} +O(2^{\omega(q)}\varphi^{*}(q)T(\log qT)^3(\log\log 3q)^5),
\end{eqnarray*}
where $\varphi^{*}(q)$ is the number of primitive characters modulo
$q$ and $\omega(q)$ is the number of distinct prime factors of $q$. 
This can only give an asymptotic relation when $2^{\omega(q)}\le\log
q$, which holds for some values of $q$, but not others.  Finally we
mention the work of Wang [\textbf{\ref{Wang}}], where an asymptotic
formula is proved for $q\le T^{1-\delta}$, for any fixed $\delta>0$.

The goal of
the present note is to establish an asymptotic formula, valid
for all $q,T\ge 2$, as soon as $q\rightarrow\infty$.

\newtheorem{theo}{Theorem}\begin{theo}
For $q,T\ge 2$ we have, in the notation above,
\begin{eqnarray*}
&&\sum_{\chi(\emph{mod}\ q)}{\!\!\!\!\!\!}^{\textstyle{*}}\
\int_{0}^{T}|L({\scriptstyle{\frac{1}{2}}}+it,\chi)|^{4}dt\\
&&\qquad=\bigg(1+O\bigg(\frac{\omega(q)}{\log
  q}\sqrt{\frac{q}{\varphi(q)}}\bigg)\bigg)\frac{\varphi^{*}(q)T}{2\pi^{2}}\prod_{p|q}
\frac{(1-p^{-1})^3}{(1+p^{-1})}(\log qT)^{4}
+O(qT(\log qT)^{\frac{7}{2}}).
\end{eqnarray*}
\end{theo}
Our proof uses ideas from the works of 
Heath-Brown [\textbf{\ref{H-B1}}] and Soundararajan
[\textbf{\ref{S}}], but there is extra work to do to handle the
integration over $t$.
\newtheorem{rem}{Remark}
\begin{rem}
\emph{It is possible, with only a little more effort, to extend the 
range to cover all $T>0$. In this case the term $\varphi^*(q)T$ in the 
main term remains the same, as does the factor $qT$ in the error term, but one must replace $\log qT$ by $\log 
q(T+2)$ both in the main term and in the error term.} 
\end{rem}
\begin{rem}
\emph{One may readily verify that our result provides an asymptotic formula, as soon as $q\rightarrow\infty$,
with an error term which saves at least a factor $O((\log\log q)^{-1/2})$.}
\end{rem}
\begin{rem}
\emph{The literature appears not to contain a precise analogue of this for
  the second moment.  However Motohashi [\textbf{\ref{Mo}}] has
  considered a uniform mean value in $t$-aspect.  He
  proved that if $\chi$ is a primitive character modulo a prime $q$, then 
\begin{equation*}
\int_{0}^{T}|L({\scriptstyle{\frac{1}{2}}}+it,\chi)|^{2}dt=
\frac{\varphi(q)T}{q}\bigg(\log\frac{qT}{2\pi}+2\gamma+2\sum_{p|q}\frac{\log 
  p}{p-1}\bigg)+
O((q^{\frac{1}{3}}T^{\frac{1}{3}}+q^{\frac{1}{2}})(\log qT)^4), 
\end{equation*}
for $T\ge 2$. This provides an asymptotic formula when $q\le
T^{2-\delta}$, for any fixed $\delta>0$. Our theorem does not give a power
saving in the error term, but it yields an asymptotic formula without
any restrictions on $q$ and $T$.} 
\end{rem}

\section{Auxiliary lemmas}

\newtheorem{lemm}{Lemma}\begin{lemm}
Let $\chi$ be a primitive character $(\emph{mod}\ q)$ such that
$\chi(-1)=(-1)^{\mathfrak{a}}$ with $\mathfrak{a}=0$ or 1.  Then we have 
\begin{equation*}
|L({\scriptstyle{\frac{1}{2}}}+it,\chi)|^2=
2\sum_{a,b\geq1}\frac{\chi(a)\overline{\chi(b)}}
{\sqrt{ab}}\bigg(\frac{a}{b}\bigg)^{-it}W_{\mathfrak{a}}\bigg(\frac{\pi
  ab}{q};t\bigg), 
\end{equation*}
where
\begin{equation*}
W_{\mathfrak{a}}(x;t)=\frac{1}{2\pi
  i}\int_{(2)}
\frac{\Gamma(\frac{1}{4}+\frac{it}{2}+\frac{z}{2}+\frac{\mathfrak{a}}{2})
\Gamma(\frac{1}{4}-\frac{it}{2}+\frac{z}{2}+\frac{\mathfrak{a}}{2})}
{|\Gamma(\frac{1}{4}+\frac{it}{2}+\frac{\mathfrak{a}}{2})|^2}
e^{z^2}x^{-z}\frac{dz}{z}.
\end{equation*}
\end{lemm} 
\begin{proof}
Let
\begin{equation*}
I:=\frac{1}{2\pi i}\int_{(2)}
\frac{\Lambda({\scriptstyle{\frac{1}{2}}}+it+z,\chi)
\Lambda({\scriptstyle{\frac{1}{2}}}-it+z,\overline\chi)}
{|\Gamma(\frac{1}{4}+\frac{it}{2}+\frac{\mathfrak{a}}{2})|^2}e^{z^2}
\frac{dz}{z}, 
\end{equation*}
where
\begin{equation*}
\Lambda({\scriptstyle{\frac{1}{2}}}+s,\chi)=
\bigg(\frac{q}{\pi}\bigg)^{s/2}\Gamma
\bigg(\frac{1}{4}+\frac{s}{2}+\frac{\mathfrak{a}}{2}\bigg)
L({\scriptstyle{\frac{1}{2}}}+s,\chi).
\end{equation*}
We recall the functional equation
\begin{equation*}
\Lambda({\scriptstyle{\frac{1}{2}}}+s,\chi)=
\frac{\tau(\chi)}{i^{\mathfrak{a}}\sqrt{q}}
\Lambda({\scriptstyle{\frac{1}{2}}}-s,\overline{\chi}).
\end{equation*}
Hence, moving the line of integration to $\Re z=-2$ and applying
Cauchy's Theorem, we obtain $|L(\frac{1}{2}+it,\chi)|^2=2I$. Finally,
expanding
$L({\scriptstyle{\frac{1}{2}}}+it+z,\chi)
L({\scriptstyle{\frac{1}{2}}}-it+z,\overline\chi)$
in a Dirichlet series and integrating termwise we obtain the lemma. 
\end{proof}

We decompose $|L(\frac{1}{2}+it,\chi)|^2$ as $2(A(t,\chi)+B(t,\chi))$, where
\begin{equation*}
A(t,\chi)=\sum_{ab\leq Z}\frac{\chi(a)\overline{\chi(b)}}{\sqrt{ab}}
\bigg(\frac{a}{b}\bigg)^{-it}W_{\mathfrak{a}}\bigg(\frac{\pi ab}{q};t\bigg),
\end{equation*}
and
\begin{equation*}
B(t,\chi)=\sum_{ab>Z}\frac{\chi(a)\overline{\chi(b)}}{\sqrt{ab}}
\bigg(\frac{a}{b}\bigg)^{-it}W_{\mathfrak{a}}\bigg(\frac{\pi ab}{q};t\bigg),
\end{equation*}
with $Z=qT/2^{\omega(q)}$. In the next two sections, we evaluate the
second moments of $A(t,\chi)$ and $B(t,\chi)$ after which our theorem 
will be an easy consequence. 

The function $W_{\mathfrak{a}}(x;t)$ approximates the characteristic
function of the interval $[0,|t|]$. Indeed, we have the following. 

\begin{lemm}
The function $W_{\mathfrak{a}}(x;t)$ satisfies
\begin{equation*}
W_{\mathfrak{a}}(x;t)=\left\{ \begin{array}{ll}
O((\tau/x)^{2}) &\qquad \textrm{for }\;\; x\geq \tau,\\
1+O((x/\tau)^{1/4}) &\qquad \textrm{for }\;\; 0<x<\tau,
\end{array} \right.
\end{equation*}
and
\begin{equation*}
\frac{\partial}{\partial t}W_{\mathfrak{a}}(x;t)\ll
\left\{ \begin{array}{ll}
\tau^{-1}(\tau/x)^{2} &\qquad \textrm{for }\;\; x\geq \tau,\\
\tau^{-1}(x/\tau)^{1/4} &\qquad \textrm{for }\;\; 0<x<\tau,
\end{array} \right.
\end{equation*}
where $\tau=|t|+2$.
\end{lemm}
\begin{proof}
The first estimate is a direct consequence of Stirling's formula,
while for the second one merely shifts the line of
integration to $\Re z=-1/4$ before employing Stirling's formula.
To handle the derivative one proceeds as before, differentiates under the 
integral sign and uses the estimate
\[\frac{\Gamma'(w)}{\Gamma(w)}=\log w+O(|w|^{-1}),\]
which holds for $1/8\le\Re w\le 2$ 
\end{proof}

The next lemma concerns the orthogonality of primitive Dirichlet
characters.

\begin{lemm}
For $(mn,q)=1$, we have
\begin{equation*}
\sum_{\chi(\emph{mod}\ q)}{\!\!\!\!\!\!}^{\textstyle{*}}\ \ 
\chi(m)\overline\chi(n)=\sum_{k|(q,m-n)}\varphi(k)\mu(q/k).
\end{equation*}
Moreover
\begin{equation*}
\sum_{\substack{\chi(\emph{mod}\ q)\\\chi(-1)=
(-1)^{\mathfrak{a}}}}{\!\!\!\!\!\!\!\!\!\!}^{\textstyle{*}}\ \ 
\chi(m)\overline\chi(n)=\frac{1}{2}\sum_{k|(q,m-n)}
\varphi(k)\mu(q/k)+\frac{(-1)^{\mathfrak{a}}}{2}
\sum_{k|(q,m+n)}\varphi(k)\mu(q/k).
\end{equation*}
\end{lemm}
\begin{proof}
This follows from [\textbf{\ref{H-B1}}; page 27]. 
\end{proof}

To handle the off-diagonal term we shall use the following bounds.
\begin{lemm}
Let $k$ be a positive integer and $Z_1,Z_2\geq2$. If 
$Z_1Z_2\leq k^{\frac{19}{10}}$ then
\begin{equation*}
E:=\sum_{\substack{Z_{1}\leq ab<2Z_{1}\\Z_{2}\leq
    cd<2Z_{2}\\ac\equiv\pm bd(\emph{mod}\ k)\\ac\ne bd\\(abcd,k)=1}}
\frac{1}{|\log\frac{ac}{bd}|}\ll\frac{(Z_{1}Z_{2})^{1+\varepsilon}}{k}
\end{equation*}
for any fixed $\varepsilon>0$, while if $Z_{1}Z_{2}>k^{\frac{19}{10}}$ then
\begin{equation}\label{2}
E\ll\frac{Z_{1}Z_{2}}{k}(\log Z_{1}Z_{2})^3.
\end{equation} 
\end{lemm}
\begin{proof}
We note that in each case the contribution of the terms with
$|\log ac/bd|>\log 2$ is satisfactory, by the corresponding lemma of
Soundararajan
[\textbf{\ref{S}}; Lemma 3]. Thus, by symmetry, it is enough to
consider the terms with $bd<ac\leq2bd$. We shall show how to handle
the terms in which $ac\equiv bd\,({\rm mod}\ k)$, the alternative case
being dealt with similarly.  We write $n=bd$ and $ac=kl+bd$ and
observe that $kl\leq bd$. We deduce that $n\leq2\sqrt{Z_1Z_2}$ and
$1\leq l\leq2\sqrt{Z_1Z_2}/k$.  Since $\log ac/bd\gg kl/n$ 
the contribution of these terms to $E$ is 
\begin{equation*}
\ll\frac{1}{k}\sum_{l\leq 2\sqrt{Z_{1}Z_{2}}/k}
\frac{1}{l}\sum_{\substack{n\leq2\sqrt{Z_{1}Z_{2}}\\(n,k)=1}}nd(n)d(kl+n).
\end{equation*}
We estimate the sum over $n$ using a bound from Heath-Brown's
paper [\textbf{\ref{H-B1}}; (17)]. This shows that the above expression is 
\begin{equation*}
\ll\frac{Z_1Z_2(\log Z_1Z_2)^2}{k}\sum_{l\leq 2\sqrt{Z_{1}Z_{2}}/k}
\frac{1}{l}\sum_{d|l}d^{-1}\ll\frac{Z_1Z_2}{k}(\log Z_1Z_2)^3.
\end{equation*}
This suffices to complete the proof.  The reader will observe that
when $Z_1Z_2\leq k^{\frac{19}{10}}$ it
is only the terms with $|\log ac/bd|>\log 2$ which prevent us from
achieving the bound \eqref{2}.
\end{proof}

Finally we shall require the following two lemmas [\textbf{\ref{S}};
Lemmas 4 and 5].

\begin{lemm}
For $q\ge 2$ we have
\begin{equation*}
\sum_{\substack{n\leq x\\(n,q)=1}}\frac{1}{n}=
\frac{\varphi(q)}{q}\big(\log x+
O\big(1+\log \omega(q)\big)\big)+O\bigg(\frac{2^{\omega(q)}\log x}{x}\bigg).
\end{equation*}
\end{lemm}

\begin{lemm}
For $x\geq\sqrt{q}$ we have
\begin{equation*}
\sum_{\substack{n\leq x\\(n,q)=1}}\frac{2^{\omega(n)}}{n}\ll
\bigg(\frac{\varphi(q)}{q}\bigg)^2(\log x)^2,
\end{equation*}
and
\begin{equation*}
\sum_{\substack{n\leq x\\(n,q)=1}}\frac{2^{\omega(n)}}{n}
\bigg(\log\frac{x}{n}\bigg)^2=
\bigg(1+O\bigg(\frac{1+\log \omega(q)}{\log q}\bigg)\bigg)
\frac{(\log x)^4}{12\zeta(2)}\prod_{p|q}\frac{1-1/p}{1+1/p}.
\end{equation*}
\end{lemm}

\section{The main term}

Applying Lemma 3 we have
\begin{equation*}
\sum_{\chi(\textrm{mod}\ q)}{\!\!\!\!\!\!}^{\textstyle{*}}\ 
\int_{0}^{T}A(t,\chi)^2dt=M+E,
\end{equation*}
where
\begin{equation*}
M=\frac{\varphi^{*}(q)}{2}\sum_{\mathfrak{a}=0,1}
\sum_{\substack{ab,cd\leq Z\\ac=bd\\(abcd,q)=1}}
\frac{1}{\sqrt{abcd}}\int_{0}^{T}W_{\mathfrak{a}}
\bigg(\frac{\pi ab}{q};t\bigg)W_{\mathfrak{a}}\bigg(\frac{\pi cd}{q};t\bigg)dt,
\end{equation*}
and
\begin{equation*}
E=\sum_{k|q}\varphi(k)\mu(q/k)E(k),
\end{equation*}
with
\begin{equation*}
E(k)=\sum_{\mathfrak{a}=0,1}
\sum_{\substack{ab,cd\leq Z\\ac\equiv\pm bd(\textrm{mod}\ k)\\ac\ne 
bd\\(abcd,q)=1}}\frac{1}{\sqrt{abcd}}\int_{0}^{T}
\bigg(\frac{ac}{bd}\bigg)^{-it}W_{\mathfrak{a}}
\bigg(\frac{\pi ab}{q};t\bigg)W_{\mathfrak{a}}\bigg(\frac{\pi cd}{q};t\bigg)dt.
\end{equation*}

We first estimate the error term $E$. We integrate by parts, using 
Lemma 2.  This produces
\begin{equation*}
E(k)\ll\sum_{\substack{ab,cd\leq Z\\ac\equiv\pm 
bd(\textrm{mod}\ k)\\ac\ne bd\\(abcd,q)=1}}
\frac{1}{\sqrt{abcd}|\log\frac{ac}{bd}|}.
\end{equation*}
We divide the terms $ab,cd\leq Z$ into dyadic blocks 
$Z_1\leq ab<2Z_1$ and $Z_2\leq cd<2Z_2$. From Lemma 4, the
contribution of this range to $E(k)$ is 
\begin{equation*}
\ll\frac{1}{\sqrt{Z_1Z_2}}\frac{Z_{1}Z_{2}}{k}(\log Z_{1}Z_{2})^3=
\frac{\sqrt{Z_{1}Z_{2}}}{k}(\log Z_1Z_2)^3,
\end{equation*}
if $Z_1Z_2>k^{\frac{19}{10}}$, and is 
$O((Z_1Z_2)^{\frac{1}{2}+\varepsilon}k^{-1})$ if $Z_1Z_2\leq
k^{\frac{19}{10}}$. Summing over all such dyadic blocks we have 
\begin{equation*}
E(k)\ll\frac{Z}{k}(\log Z)^3+k^{-\frac{1}{20}+2\varepsilon}.
\end{equation*}
Thus
\begin{equation}\label{3}
E\ll Z2^{\omega(q)}(\log Z)^3\ll qT(\log qT)^3.
\end{equation}

We now turn to the main term $M$. Since $ac=bd$, we can write $a=gr$,
$b=gs$, $c=hs$ and $d=hr$, where $(r,s)=1$. We put $n=rs$. Hence 
\begin{equation*}
M=\frac{\varphi^{*}(q)}{2}\sum_{\mathfrak{a}=0,1}
\sum_{\substack{n\leq Z\\(n,q)=1}}\frac{2^{\omega(n)}}{n}
\sum_{\substack{g,h\leq \sqrt{Z/n}\\(gh,q)=1}}\frac{1}{gh}
\int_{0}^{T}W_{\mathfrak{a}}\bigg(\frac{\pi g^2n}{q};t\bigg)
W_{\mathfrak{a}}\bigg(\frac{\pi h^2n}{q};t\bigg)dt.
\end{equation*}
From Lemma 2 we have $W_{\mathfrak{a}}(\pi g^2n/q;t)=1+
O(g^{1/2}(n/qt)^{\frac{1}{4}})$, whence
\begin{equation*}
M=\varphi^{*}(q)T\sum_{\substack{n\leq Z\\(n,q)=1}}
\frac{2^{\omega(n)}}{n}\bigg(\sum_{\substack{g\leq \sqrt{Z/n}\\(g,q)=1}}
\frac{1}{g}+O(1)\bigg)^2.
\end{equation*}

We split the terms $n\leq Z$ into the cases $n\leq Z_0$ and $Z_0<n\leq
Z$, where $Z_0=Z/9^{\omega(q)}$. In the first case, from Lemma 5 the sum
over $g$ is 
\begin{equation*}
=\frac{\varphi(q)}{2q}\log\frac{Z_0}{n}+O(1+\log \omega(q)),
\end{equation*}
since the first error term in Lemma 5 dominates the second.
Hence the contribution of such values of $n$ to $M$ is
\begin{equation*}
\varphi^{*}(q)T\bigg(\frac{\varphi(q)}{2q}\bigg)^2
\sum_{\substack{n\leq Z_0\\(n,q)=1}}\frac{2^{\omega(n)}}{n}
\bigg(\bigg(\log\frac{Z_0}{n}\bigg)^2+O(\omega(q)\log Z)\bigg).
\end{equation*}
Here we use the fact that $q/\varphi(q)\ll1+\log\,\omega(q)$.  This estimate will
be employed a number of times in what follows, without further comment. In view of Lemma 6 the contribution from terms with $n\le Z_0$ is now
seen to be 
\begin{equation}\label{4}
\frac{\varphi^{*}(q)T}{8\pi^2}\prod_{p|q}
\frac{(1-1/p)^3}{(1+1/p)}(\log Z_0)^4
\bigg(1+O\bigg(\frac{\omega(q)}{\log q}\bigg)\bigg).
\end{equation}

For $Z_0\leq n\leq Z$, we extend the sum over $g$ to all
$g\leq3^{\omega(q)}$ that are coprime to $q$. By Lemma 5, this sum is $\ll
\omega(q)\varphi(q)/q$. Hence the contribution of these terms to $M$ is 
\begin{equation*}
\ll\varphi^{*}(q)T\bigg(\omega(q)\frac{\varphi(q)}{q}\bigg)^2
\sum_{Z_0\leq n\leq Z}\frac{2^{\omega(n)}}{n}\ll \varphi^{*}(q)T
\bigg(\frac{\varphi(q)}{q}\bigg)^4\omega(q)^2(\log Z)^2.
\end{equation*}
Combining this with \eqref{3} and \eqref{4} we obtain
\begin{equation}\label{7}
\sum_{\chi(\textrm{mod}\ q)}{\!\!\!\!\!\!}^{\textstyle{*}}\ 
\int_{0}^{T}A(t,\chi)^2dt=\bigg(1+O\bigg(\frac{\omega(q)}{\log q}\bigg)\bigg)\frac{\varphi^{*}(q)T}{8\pi^2}
\prod_{p|q}\frac{(1-1/p)^3}{(1+1/p)}(\log qT)^4
.
\end{equation}

\section{The error term}

We have
\begin{eqnarray}\label{9}
\sum_{\chi(\textrm{mod}\ q)}{\!\!\!\!\!\!}^{\textstyle{*}}\ 
\int_{0}^{T}B(t,\chi)^2dt&\leq&\sum_{\chi(\textrm{mod}\ q)}
\int_{0}^{T}B(t,\chi)^2dt\nonumber\\
&&\!\!\!\!\!\!\!\!\!\!\!\!\!\!\!\!\!\!\!\!\!\!\!\!\!\!\!\!\!\!
\!\!\!\!\!\!\!\!\!\!\!\!\!\!\!\!\!\!\!\!\!\!\!\!\!\!\!\!\!\!\!\!\!\!\!
=\frac{\varphi(q)}{2}\sum_{\mathfrak{a}=0,1}
\sum_{\substack{ab,cd>Z\\ac\equiv\pm bd(\textrm{mod}\ q)\\(abcd,q)=1}}
\frac{1}{\sqrt{abcd}}\int_{0}^{T}\bigg(\frac{ac}{bd}\bigg)^{-it}
W_{\mathfrak{a}}\bigg(\frac{\pi ab}{q};t\bigg)W_{\mathfrak{a}}
\bigg(\frac{\pi cd}{q};t\bigg)dt.
\end{eqnarray}
Using Lemma 2 and integration by parts, the integral over $t$ is
\begin{equation*}
\ll\frac{1}{|\log\frac{ac}{bd}|}\bigg(1+\frac{ab}{qT}\bigg)^{-2}
\bigg(1+\frac{cd}{qT}\bigg)^{-2}
\end{equation*}
if $ac\ne bd$, and is
\begin{equation*}
\ll T\bigg(1+\frac{ab}{qT}\bigg)^{-2}\bigg(1+\frac{cd}{qT}\bigg)^{-2}
\end{equation*}
if $ac=bd$. Hence the right hand side of \eqref{9} is $O(R_1+R_2)$, where
\begin{equation*}
R_1=\varphi(q)T\sum_{\substack{ab,cd>Z\\ac=bd\\(abcd,q)=1}}
\frac{1}{\sqrt{abcd}}\bigg(1+\frac{ab}{qT}\bigg)^{-2}
\bigg(1+\frac{cd}{qT}\bigg)^{-2},
\end{equation*}
and
\begin{equation*}
R_2=\varphi(q)\sum_{\substack{ab,cd>Z\\ac\equiv\pm 
bd(\textrm{mod}\ q)\\ac\ne
bd\\(abcd,q)=1}}\frac{1}{\sqrt{abcd}|\log\frac{ac}{bd}|}
\bigg(1+\frac{ab}{qT}\bigg)^{-2}
\bigg(1+\frac{cd}{qT}\bigg)^{-2}.
\end{equation*}

To estimate $R_2$, we again break the terms into dyadic ranges
$Z_1\leq ab<2Z_1$ and $Z_2\leq cd<2Z_2$, where $Z_1,Z_2>Z$. By Lemma 4, the
contribution of each such block is 
\begin{equation*}
\ll\frac{\varphi(q)}{\sqrt{Z_1Z_2}}\bigg(1+\frac{Z_1}{qT}\bigg)^{-2}
\bigg(1+\frac{Z_2}{qT}\bigg)^{-2}\frac{Z_1Z_2}{q}(\log Z_1Z_2)^3.
\end{equation*}
Summing over all the dyadic ranges we obtain
\begin{equation}\label{10}
R_2\ll \varphi(q)T(\log qT)^3.
\end{equation}

To handle $R_1$ we argue as in the previous section.  We write
$a=gr$, $b=gs$, $c=hs$ and $d=hr$, where $(r,s)=1$, and we put
$n=rs$. Then 
\begin{equation}\label{5}
R_1\ll\varphi(q)T\sum_{(n,q)=1}\frac{2^{\omega(n)}}{n}
\bigg(\sum_{\substack{g>\sqrt{Z/n}\\ (g,q)=1}}
\frac{1}{g}\bigg(1+\frac{g^2n}{qT}\bigg)^{-2}\bigg)^2.
\end{equation}

We split the sum over $n$ into the ranges $n\leq qT$ and $n>qT$. In the
first case, the sum over $g$ is 
\begin{equation*}
\ll1+\sum_{\substack{\sqrt{Z/n}\leq g\leq\sqrt{qT/n}\\(g,q)=1}}\frac{1}{g}.
\end{equation*}
When $n\leq Z_0$ this is 
\begin{equation*}
\ll \frac{\varphi(q)}{q}\,\omega(q).
\end{equation*}
by Lemma 5.  In the alternative case $n>Z_0$ we extend the sum over
$g$ to include all $g\leq 3^{\omega(q)}$ that are coprime to $q$.
Lemma 5 then gives the same bound as before. Thus the contribution of 
the terms $n\leq qT$ to \eqref{5}, using Lemma 6, is 
\begin{equation}\label{6}
\ll\varphi(q)T\bigg(\frac{\varphi(q)}{q}\omega(q)\bigg)^2
\sum_{\substack{n\leq qT\\(n,q)=1}}\frac{2^{\omega(n)}}{n}\ll 
qT\bigg(\frac{\varphi(q)}{q}\bigg)^5\omega(q)^2(\log qT)^2.
\end{equation}

In the remaining case $n>qT$, the sum over $g$ in \eqref{5} is 
$O(q^2T^2/n^2)$. Hence the contribution of such terms is 
\begin{equation*}
\ll\varphi(q)T\sum_{n>qT}\frac{2^{\omega(n)}}{n}\frac{q^4T^4}{n^4}\ll
\varphi(q)T\log qT.
\end{equation*}
In view of \eqref{10} and \eqref{6} we now have
\begin{equation}\label{8}
\sum_{\chi(\textrm{mod}\ q)}{\!\!\!\!\!\!}^{\textstyle{*}}\ 
\int_{0}^{T}B(t,\chi)^2dt\ll qT\bigg(\frac{\varphi(q)}{q}\bigg)^5
\omega(q)^2(\log qT)^2+\varphi(q)T(\log qT)^3.
\end{equation}

\section{Deduction of Theorem 1}

From Lemma 1 we have
\begin{equation*}
\sum_{\chi(\textrm{mod}\ q)}{\!\!\!\!\!\!}^{\textstyle{*}}\ 
\int_{0}^{T}|L({\scriptstyle{\frac{1}{2}}}+it,\chi)|^{4}dt=
4\sum_{\chi(\textrm{mod}\ q)}{\!\!\!\!\!\!}^{\textstyle{*}}\ 
\int_{0}^{T}\big(A(t,\chi)^2+2A(t,\chi)B(t,\chi)+B(t,\chi)^2\big)dt.
\end{equation*}
The first and third terms on the right hand side are handled by \eqref{7} 
and \eqref{8}. Also, by Cauchy's inequality we have
\begin{equation*}
\sum_{\chi(\textrm{mod}\ q)}{\!\!\!\!\!\!}^{\textstyle{*}}\ 
\int_{0}^{T}A(t,\chi)B(t,\chi)dt\leq
\bigg(\sum_{\chi(\textrm{mod}\ q)}{\!\!\!\!\!\!}^{\textstyle{*}}\ 
\int_{0}^{T}A(t,\chi)^2dt\bigg)^{\frac{1}{2}}
\bigg(\sum_{\chi(\textrm{mod}\ q)}{\!\!\!\!\!\!}^{\textstyle{*}}\ 
\int_{0}^{T}B(t,\chi)^2dt\bigg)^{\frac{1}{2}}.
\end{equation*}
Hence \eqref{7} and \eqref{8} also yield an estimate for the cross
term. Combining these results leads to the theorem.

\end{document}